\documentclass{article}
\usepackage[colorinlistoftodos]{todonotes}   
\usepackage{amsmath,amssymb,amsthm}
\usepackage{graphicx} % Required for inserting images\usepackage{amsmath,amssymb,amsthm}
\usepackage{mathrsfs}
\usepackage{hyperref}
\usepackage{pgf,tikz,pgfplots}
\usepackage[utf8]{inputenc}
\usepackage{mathrsfs}
\usepackage{mathtools}
\usepackage[colorinlistoftodos]{todonotes}   
\usepackage{amsmath}
\usepackage{amssymb}
\usepackage{amsthm}
\usepackage{lmodern}
\numberwithin{equation}{section}

\usepackage[left=3cm, right=3cm]{geometry}
\pgfplotsset{compat=1.15}

\usetikzlibrary{arrows}
\usepackage{array}
\usepackage{color, marginnote}

\newcommand{\RR}{\mathbb{R}}

\newcommand{\DD}{\mathbb{D}}

\newtheorem{theorem}{Theorem}%[section]
\theoremstyle{theorem}
\newtheorem{proposition}[theorem]{Proposition}
\newtheorem{definition}[theorem]{Definition}
\newtheorem{lemma}[theorem]{Lemma}
\newtheorem{corollary}[theorem]{Corollary}
\newtheorem{remark}[theorem]{Remark}

\title{Starting the study of outer length billiards}

\author{Luca Baracco, Olga Bernardi, Corentin Fierobe \\ \\
baracco@math.unipd.it, obern@math.unipd.it, cpef@gmx.de}

\date{ }

\begin{document}

\maketitle

\begin{abstract}
\noindent We focus on the outer length billiard dynamics, acting on the exterior of a strictly-convex planar domain. We first show that ellipses are totally integrable. We then provide an explicit representation of first order terms for the formal Taylor expansion of the corresponding Mather's $\beta$-function. Finally, we provide explicit Lazutkin coordinates up to order 4.
\end{abstract}
\section{Introduction}
The aim of the present paper is starting an accurate study of outer length billiards, first described by P. Albers and S. Tabachnikov in 2024, see \cite{AT}[Section 3.4]. These billiards are the counterpart of Birkhoff ones since the generating function is the outer length instead of the inner length. They are also called ``fourth billiards''. In fact, there are two billiards systems --Birkhoff and outer area billiards-- which have been extensively studied in literature; we refer respectively to \cite{TabLIBRO} and \cite{Tab-outer} for exhaustive surveys. Another type of billiards, namely symplectic billiards, whose generating function is the inner area, were introduced in 2018 by P. Albers and S. Tabachnikov \cite{ATSym} and their study started to become more intensive only recently. We refer to \cite{BaBe}, \cite{BaBeNa4} and \cite{Daniel} for integrability results and to \cite{BaBeNa1} and \cite{CorVigAlf} for area spectral rigidity results for symplectic billiards. Regarding outer length billiards, to the best of our knowledge, they were not studied yet. However, the seminal idea on the base of the definition of this dynamical system --detecting, in particular, circumscribed polygons to a strictly-convex domain with minimal perimeter-- can already be found in some former papers of convex planar geometry, see e.g. \cite{DeTe}[Theorem 1] and \cite{CiGo}[Section 2]. \\
~\newline
\noindent We first give all the details to introduce this dynamical system, acting on the exterior of a strictly-convex planar domain. We then prove --by arguments of elementary planar geometry-- that ellipses are totally integrable, that is the phase-space is fully foliated by continuous invariant curves which are not null-homotopic. \\
~\newline
\noindent We successively focus on the main topic of the paper, which is providing an explicit representation of first order terms for the formal Taylor expansion of Mather's $\beta$-function (or minimal average function) for outer length billiards. In particular, we write these coefficients (up to order $5$) by means of the ordinary curvature and length of the boundary of the billiard table. As already noticed, for such a dynamical system, Mather's $\beta$-function is related to the minimal perimeter of polygons circumscribed to a strictly-convex domain. These perimeters are special cases (i.e. for periodic trajectories of winding number $=1$) of the corresponding marked length spectrum for outer length billiards. \\
~\newline
\noindent Finally, by using the computations we made to obtain minimal average function's coefficients, we provide explicit Lazutkin coordinates up to order 4 and discuss straightforward facts regarding the existence/non existence of caustics for outer length billiards. \\
~\newline
\noindent In order to state our results, we proceed with some preliminaries.
\section{Twist maps and Mather's $\beta$-function}

Let $\mathbb{T} \times (a,b)$ be the annulus, where $\mathbb{T} = \mathbb{R}/ \mathbb{Z} = [0, 1]/ \sim$ identifying $0 \sim 1$ and (eventually) $a=-\infty$ and/or $b= +\infty$. Given a diffeomorphism $\Phi: \mathbb{S}^1 \times (a,b) \rightarrow \mathbb{S}^1 \times (a,b),$ we denote by 
$$\phi: \mathbb{R} \times (a,b) \rightarrow \mathbb{R} \times (a,b), \qquad (x_0,y_0) \mapsto (x_1,y_1)$$
a lift of $\Phi$ to the universal cover. Then $\phi$ is a diffeomorphism and $\phi(x+1,y) = \phi(x,y) + (1,0)$. In the case where $a$ (resp. $b$) is finite, we assume that $\phi$ extends continuously to $\mathbb{R} \times \{a\}$ (resp. $\mathbb{R} \times \{b\}$) by a rotation of fixed angle:
\begin{equation} \label{rotazione}
\phi(x_0,a) = (x_0 + \rho_a,a) \qquad (\text{resp. \ } \phi(x_0,b) = (x_0 + \rho_b,b)).
\end{equation}
Once fixed the lift, the numbers $\rho_a, \rho_b$ are unique. The choice of $\rho_a$ (resp. $\rho_b$) if $a = -\infty$ (resp. $b = +\infty$) depends on the dynamics at infinity. For example, in the case of outer length billiards where $b = +\infty$, it is natural to set $\rho_b = 1/2$, we refer to point 1. of Section \ref{SIS-DIN} for details. \\
\indent We recall here below -- for reader's convenience -- the well-known definition of monotone twist map, see e.g. \cite{Si}[Page 2]. 
\begin{definition}
A monotone twist map $\phi: \mathbb{R} \times (a,b) \rightarrow \mathbb{R} \times (a,b)$, $(x_0,y_0) \mapsto (x_1,y_1)$
is a diffeomorphism satisfying:
\begin{enumerate}
\item $\phi(x_0+1,y_0) = \phi(x_0,y_0) + (1,0)$.
\item $\phi$ preserves orientations and the boundaries of $\mathbb{R} \times (a,b)$. 
\item $\phi$ extends to the boundaries by rotation, as in (\ref{rotazione}).
\item $\phi$ satisfies a monotone twist condition, that is
\begin{equation} \label{TC}
    \frac{\partial x_1}{\partial y_0} > 0.
\end{equation}
\item $\phi$ is exact symplectic; this means that there exists a generating function $H \in C^2(\mathbb{R} \times \mathbb{R}; \mathbb{R})$ for $\phi$ such that
\begin{equation} \label{GF}
y_1dx_1 - y_0dx_0 = dH(x_0,x_1).
\end{equation}
\end{enumerate}
\end{definition}
\noindent Clearly, $H(x_0+1,x_1+1) = H(x_0,x_1)$ and, due to the twist condition, the domain of $H$ is the strip $\{(x_0,x_1): \ \rho_a + x_0 < x_1 < x_0 + \rho_b\}$. Moreover, equality (\ref{GF}) reads
\begin{equation} \label{ESSE}
\begin{cases} y_1=H_2(x_0,x_1)\\
y_0=-H_1(x_0,x_1)
\end{cases}
\end{equation}
and the twist condition (\ref{TC}) becomes $H_{12} < 0$. As a consequence, $\{(x_i,y_i)\}_{i \in \mathbb{Z}}$ is an orbit of $\phi$ if and only if $H_2(x_{i-1},x_i) = y_i = -H_1(x_i,x_{i+1})$ for all $i \in \mathbb{Z}$. This means -- on a formal level -- that the corresponding bi-infinite sequence $x := \{x_i\}_{i \in \mathbb{Z}}$ is a so-called critical configuration of the action functional $\sum_{i\in \mathbb{Z}} H(x_i,x_{i+1})$. In such a setting, minimal orbits play a fundamental role. We recall that a critical configuration $x$ of $\phi$ is minimal if every finite segment of $x$ minimizes the action functional with fixed end points (we refer to \cite{Si}[Page 7] for details). Clearly, all these facts remain true if we consider a monotone twist map on $\{(x_0,x_1): \ u_a(x_0) < x_1 < u_b(x_0)\}$. \\
\indent For a twist map $\phi$ generated by $H$, we finally introduce the rotation number and the Mather's $\beta$-function (or minimal average action).
\begin{definition} The rotation number of an orbit $\{(x_i,y_i)\}_{i \in \mathbb{Z}}$ of $\phi$ is 
$$\rho := \lim_{i\to\pm\infty}\frac{x_i}{i}$$
if such a limit exists. 
\end{definition}
\noindent A relevant class of monotone twist maps are planar billiard maps. In such a setting, the rotation number of a periodic trajectory is the rational 
$$\frac{m}{n} = \frac{\text{winding number}}{\text{number of reflections}} \in (0, \frac{1}{2}\Big{]},$$ 
\noindent we refer to \cite{Si}[Page 40] for details. \\
\indent In view of the celebrated Aubry-Mather theory (see e.g. \cite{Bang}), a monotone twist map possesses minimal orbits for every rotation number $\rho$ inside the so-called twist interval $(\rho_{a},\rho_{b})$. As a consequence, we can associate to each $\rho$ the average action of any minimal orbit having that rotation number.  
\begin{definition} \label{Mather beta}
The Mather's $\beta$-function of $\phi$ is $\beta: (\rho_a,\rho_b) \to \mathbb{R}$ with
$$\beta(\rho) := \lim_{N\to\infty}\frac{1}{2N}\sum_{i=-N}^{N-1} H(x_i,x_{i+1})$$
where $\lbrace x_i\rbrace_{i\in \mathbb{Z}}$ is any minimal configuration of $\phi$ with rotation number $\rho$.
\end{definition}
In the framework of Birkhoff billiards, A. Sorrentino in \cite{Sor} gave an explicit representation of the coefficients of the (formal) Taylor expansion at zero of the corresponding Mather’s $\beta$-function. More recently, J. Zhang in \cite{Zhang} got (locally) an explicit formula for this function via a Birkhoff normal form. Moreover, M. Bialy in \cite{BiEl} obtained an explicit formula for Mather's $\beta$-function for ellipses by using a non-standard generating function, involving the support function. Regarding symplectic and outer billiards, the first two authors and A. Nardi in \cite{BaBeNa} computed explicitly the higher order terms of such an expansion, by using tools from affine differential geometry. As anticipated, one of the target of the present paper is writing explicitly these coefficients (up to order 5) in the case of forth billiards. 

\section{The dynamical system} \label{SIS-DIN}
Let $\Omega$ be a strictly-convex planar domain with smooth boundary $\partial\Omega$. Assume that the perimeter of $\partial \Omega$ is $\ell = |\partial \Omega|$. Fixed the positive counter-clockwise orientation, let $\gamma: \mathbb{T} \to \partial \Omega$ be the smooth arc-length parametrization of $\partial \Omega$. For every $s \in \mathbb{T}$, we denote by $s^* \in \mathbb{T}$ the (unique, by strict-convexity) arc-length parameter such that $T_{\gamma(s)} \partial \Omega = T_{\gamma(s^*)} \partial \Omega$. We refer to $$\mathcal P = \{(s,r) \in \mathbb{T} \times \mathbb{T}: \ s < r < s^*\}$$ as the (open, positive) phase-space and we define the outer length billiard map as follows (we refer to \cite{AT}[Section 3.4]). \\
\indent Since $\Omega$ is strictly-convex, to every point $P \in \mathbb{R}^2 \setminus \text{cl}(\Omega)$ can be uniquely associated a pair $(s_0,s_{1}) \in \mathbb{T} \times \mathbb{T}$ with $s_0 < s_{1}$ and such that the lines $P\gamma(s_0)$ and $P\gamma(s_{1})$ are the (negative and positive) tangents to $\partial \Omega$. 
Consider the circle in $\mathbb{R}^2 \setminus \Omega$ tangent to $\partial \Omega$ at $\gamma(s_{1})$ and to the line $P\gamma(s_0)$. Then then image $P'$ of $P$ is defined as the intersection point between the lines $P\gamma(s_{1})$ and the other common tangent line of the circle and $\partial \Omega$ (hence passing through $P'$ and $\gamma(s_{2})$): 
$$T : \mathcal{P} \to \mathcal{P}, \qquad (s_0,s_{1}) \mapsto (s_{1},s_{2}).$$

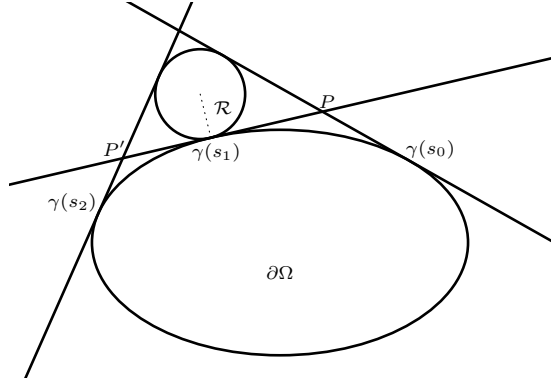
\begin{figure}
    \centering
\begin{tikzpicture}[line cap=round,line join=round,>=triangle 45,x=2cm,y=2cm]
\clip(-1.8,-0.9) rectangle (1.8,1.6);
\draw [rotate around={0:(0,0)},line width=1pt] (0,0) ellipse (2.5cm and 1.5cm);
\draw [line width=1pt,domain=-2.4644525360396767:2.925337292636209] plot(\x,{(--14.0625--10.87381925368578*\x)/4.808139752941082});
\draw [line width=1pt,domain=-2.4644525360396767:2.925337292636209] plot(\x,{(--14.0625--4.077092715913812*\x)/17.475364801934});
\draw [line width=1pt,domain=-2.4644525360396767:2.925337292636209] plot(\x,{(--14.0625-7.712987683825162*\x)/13.649615479846691});
\draw [line width=1pt] (-0.5304631168743307,0.9873081843929444) circle (0.5976cm);
\draw[black,dotted] (-0.5304631168743307,0.9873081843929444) -- (-0.46,0.7);
\begin{scriptsize}
\draw[color=black] (1,0.6221857002178954) node {$\gamma(s_0)$};
\draw[color=black] (-0.42393158262452424,0.6) node {$\gamma(s_1)$};
\draw[color=black] (-1.38,0.26685360315769346) node {$\gamma(s_2)$};
\draw[color=black] (0.3113595885199015,0.938818261954709) node {$P$};
\draw[color=black] (-1.1,0.6292219793676024) node {$P'$};
\draw[color=black] (0,-0.2) node {$\partial\Omega$};
\draw[color=black] (-0.38,0.9) node {$\mathcal R$};
\end{scriptsize}
\end{tikzpicture}
    \caption{The outer-length billiard map around the domain $\Omega$. It associates the point $P$ to the point $P'$.}
    \label{fig:outer_length_billiard_reflection}
\end{figure}
\noindent We refer to Figure \ref{fig:outer_length_billiard_reflection}. Setting $\varepsilon_0 = s_{1} - s_0$ and
$$\hat{\mathcal{P}} = \{(s,\varepsilon) \in \mathbb{T} \times \mathbb{R}: \ 0 < \varepsilon < s^*-s\},$$
the outer length billiard map can be equivalently defined as
\[
T: \hat{\mathcal{P}} \to \hat{\mathcal{P}}, \qquad (s_0,\varepsilon_0) \mapsto(s_{1},\varepsilon_{1}).
\]
Here are some properties of the outer length billiard map. \\
~\newline
1. $T$ is continuous and can be continuously extended so that $T(s,s) = (s,s)$ and $T(s,s^*) = (s^*,s).$ \\
~\newline
2. The function 
$$H: \mathcal{P} \to \mathbb{R}, \qquad H(s_0,s_{1}) := |P\gamma(s_0)| + |P\gamma(s_{1})|$$ 
generates $T$, that is
\begin{equation}
\label{var}
T(s_0,s_{1}) = (s_{1},s_{2}) \qquad \Longleftrightarrow \qquad H_2(s_0,s_{1}) + H_1(s_{1},s_{2}) = 0.
\end{equation}
We refer to \cite{AT}[Lemma 3.1] for the proof. In view of (\ref{var}), we can equivalently refer to
$$\bar{H}: \mathcal{P} \to \mathbb{R}, \qquad \bar{H}(s_0,s_{1}) := |P\gamma(s_0)| + |P\gamma(s_{1})| -s_{1} + s_0$$ 
as a generating function, which is exactly the Lazutkin parameter of $\partial \Omega$, interpreted as convex caustic for a Birkhoff billiard. \\
~\newline
3. $T$ is a twist map preserving the area form $-H_{12}(s_0,s_1) \ ds_0 \wedge ds_1$. \\
~\newline
4. By introducing new variables
$$y_0 = -H_1(s_0,s_{1}), \qquad y_{1} = H_2(s_0,s_{1}),$$
$(s,y)$ are coordinates on $\mathcal{P}$ and the outer length billiard map results a (negative) twist map, since
$$\frac{\partial y_{1}}{\partial s_0} = H_{12}(s_0,s_{1}) = - \frac{k(s_{0}) k(s_{1}) H(s_0,s_{1})} {2 \sin^2(\varphi/2)} < 0,$$
where $\varphi$ is the angle between the tangent lines $P\gamma(s_0)$ and $P\gamma(s_{1})$ (see also \cite{AT}[Page 11]). In these coordinates, the preserved area form is the standard one: $ds \wedge dy$. \\
~\newline
5. The marked length spectrum for the outer length billiard is the map $\mathcal{ML}_o(\Omega): \mathbb{Q} \cap{(0, \frac{1}{2}\Big{)}} \to \mathbb{R}$ that associates to any $m/n$ in lowest terms the minimal perimeter of the periodic trajectories having rotation number $m/n$. We refer to \cite{Si}[Sections 3.1 and 3.2] for a general treatment of the marked spectrum. Clearly, periodic outer length billiard minimal trajectories (with winding number $=1$) correspond to convex polygons realizing the minimal (circumscribed) perimeter, so that: 
\begin{equation} \label{beta-MA}
    \beta\left( \frac{1}{n}\right) = \frac{1}{n} \mathcal{ML}_o(\Omega) \left( \frac{1}{n}\right).
\end{equation}

\subsection{Circles and ellipses} %\olga{This section to be revised}

\noindent As expected, the outer length billiard on the circle (of center $O$) is totally integrable: the phase-space is completely foliated by concentric invariant circles. By using as coordinates $(\alpha_0,\alpha_{1}) \in \mathbb{T} \times \mathbb{T}$, where $\alpha_0$ and $\alpha_{1}$ are respectively the angles of $O\gamma(s_0)$ and $O\gamma(s_{1})$ with respect to the positive horizontal direction, the generating function in the case of disk of unit radius is $$H(\alpha_0,\alpha_{1}) = 2\tan\left(\frac{\alpha_{1} - \alpha_0}{2}\right).$$ 
Equivalently, in terms of $(\alpha_0,y_0) = (\alpha_0, -H_1(\alpha_0,\alpha_{1})) =     \left(\alpha_0, 1+\tan^2\left(\frac{\alpha_1-\alpha_0}{2}\right)\right)$, we have that %$H(\alpha_0,y_{0}) = 2 (\sqrt{y_0} - \arctan \sqrt{y_0})$, 

\[
H(\alpha_0,y_0) = 2\sqrt{y_0-1}
\]
and the total integrability immediately follows. \\
\indent An unexpected fact --at least from the authors' point of view, since the billiard dynamics are not invariant by affine transformations-- is that also the outer length billiard on the ellipse is totally integrable, as stated in the next proposition. 

\begin{proposition}
\label{prop:integrability_ellipse}
    Let $\mathcal E$ and $\Gamma$ be two confocal nested ellipses, $\mathcal E \subset \Gamma$. Then $\Gamma$ is a caustic for the outer-length billiard dynamics outside $\mathcal E$.
\end{proposition}

\noindent The proof of Proposition \ref{prop:integrability_ellipse} relies on the following lemma of elementary planar geometry, see Figure \ref{fig:orthogonality_ellipse}.

\begin{lemma}[Lemma 2.4 in \cite{Stachel}]
\label{lemma:orthogonality_ellipse}
    Let $P_0,P_1\in\Gamma$ two distinct points such that the line $P_0P_1$ is tangent to $\mathcal E$ at a point $Q$. Let $R$ be the intersection point of the tangent lines to $\Gamma$ at $P_0$ and $P_1$. Then the lines $P_0P_1$ and $RQ$ are orthogonal.
\end{lemma}

%\corentin{Should we give an elementary proof of this result? The one in \cite{Stachel} is computational...}

\begin{proof}[Proof of Proposition \ref{prop:integrability_ellipse}]
Let a point $P_0$ on $\Gamma$. Consider the positive tangent line to $\mathcal E$ at a point $Q$ and passing through $P_0$. %We choose the one so that the orientation from $P_0$ to $Q$ is the orientation we fixed at the beginning for the outer-length billiard. 
let $P_1\in\Gamma$ be the intersection point of the latter tangent line $P_0Q$ with $\Gamma$, see Figure \ref{fig:outer_length_ellipse}. We need to show that $P_1$ is the image of $P_0$ under the outer-length billiard reflection outside $\mathcal E$. Consider the point $P$, such that $PP_0$ and $PP_1$ are the two tangent lines to $\mathcal E$ passing through $P$, see Figure \ref{fig:outer_length_ellipse}. Since $\mathcal E$ and $\Gamma$ are confocal, $\mathcal E$ is a caustic for the classical billiard in $\Gamma$. In particular, the tangent line $T_{P_0}\Gamma$ is a bisector of the angle $\widehat{P_1P_0P}$. With the same argument the tangent line $T_{P_1}\Gamma$ is a bisector of the angle $\widehat{P_0P_1P}$. Hence $T_{P_0}\Gamma$ and $T_{P_1}\Gamma$ intersects at a point $R$ which is the center of the inscribed circle $\mathcal D$ to the triangle $P_0PP_1$. Now --by Lemma \ref{lemma:orthogonality_ellipse}-- the lines $RQ$ and $P_0P_1$ are orthogonal. In particular $\mathcal D$ is tangent to the ellipse $\mathcal E$. This implies that $P_1$ is obtained from $P_0$ by the outer-length billiard law of reflection.
\end{proof}

\begin{figure}
    \centering
    \definecolor{wrwrwr}{rgb}{0.3803921568627451,0.3803921568627451,0.3803921568627451}
\definecolor{rvwvcq}{rgb}{0.08235294117647059,0.396078431372549,0.7529411764705882}
\begin{tikzpicture}[line cap=round,line join=round,>=triangle 45,x=2cm,y=2cm, scale=0.8]
\clip(-2.2,-1.4) rectangle (2.2,2);
\draw [rotate around={0:(0,0)},line width=1pt] (0,0) ellipse (2.828427125cm and 2cm);
\draw [rotate around={0:(0,0)},line width=1pt] (0,0) ellipse (3.12409987cm and 2.4cm);
\draw [line width=1pt,domain=-3.2529192711182193:3.6970088523158386] plot(\x,{(--56.2176--33.193035646061155*\x)/18.106190780256863});
\draw [line width=1pt,domain=-3.2529192711182193:3.6970088523158386] plot(\x,{(--56.2176--3.633386971464968*\x)/46.60864613103539});
\draw [line width=1pt] (-1.440669949915849,0.4637856244942844)-- (-0.15769908730316703,1.1938690095039803);
\draw [line width=1pt] (-1.0817149356921705,1.1218371565399394)-- (-0.8866515866329271,0.779053580930514);
\begin{scriptsize}
\draw[color=black] (-0.6897930376063115,-0.75) node {$\mathcal E$};
\draw[color=black] (-0.85,-1.12) node {$\Gamma$};
\draw [fill=wrwrwr] (-0.8866515866329271,0.779053580930514) circle (1pt);
\draw[color=wrwrwr] (-0.8485707688858103,0.65) node {$Q$};
\draw[color=black] (-0.42,1.8) node {$T_{P_0}\Gamma$};
\draw[color=black] (-2,0.89) node {$T_{P_1}\Gamma$};
\draw[color=wrwrwr] (-1.15,1.23) node {$R$};
\draw [fill=rvwvcq] (-1.440669949915849,0.4637856244942844) circle (1pt);
\draw[color=wrwrwr] (-1.57,0.5) node {$P_0$};
\draw [fill=wrwrwr] (-0.15769908730316703,1.1938690095039803) circle (1pt);
\draw[color=wrwrwr] (-0.12272971160810187,1.2916510015726108) node {$P_1$};
\end{scriptsize}
\end{tikzpicture}
\caption{The line $RQ$ is orthogonal to the line $P_0P_1$}
\label{fig:orthogonality_ellipse}
\end{figure}
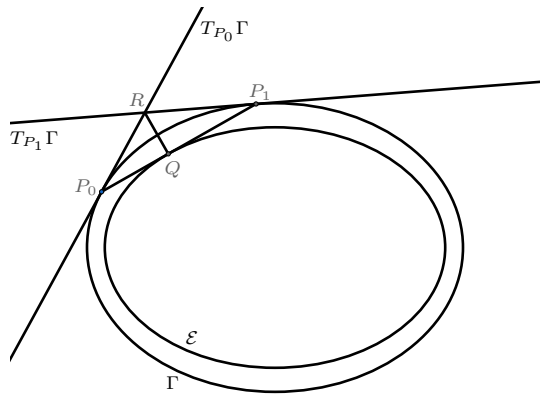

\begin{figure}
    \centering
    \definecolor{wrwrwr}{rgb}{0.3803921568627451,0.3803921568627451,0.3803921568627451}
\definecolor{rvwvcq}{rgb}{0.08235294117647059,0.396078431372549,0.7529411764705882}
\begin{tikzpicture}[line cap=round,line join=round,>=triangle 45,x=2cm,y=2cm, scale=0.8]
\clip(-2.2,-1.4) rectangle (2.2,2);
\draw [rotate around={0:(0,0)},line width=1pt] (0,0) ellipse (2.828427125cm and 2cm);
\draw [rotate around={0:(0,0)},line width=1pt] (0,0) ellipse (3.12409987cm and 2.4cm);
\draw [line width=1pt,domain=-3.2529192711182193:3.6970088523158386] plot(\x,{(--56.2176--33.193035646061155*\x)/18.106190780256863});
\draw [line width=1pt,domain=-3.2529192711182193:3.6970088523158386] plot(\x,{(--56.2176--3.633386971464968*\x)/46.60864613103539});
\draw [line width=1pt] (-1.4131545906973866,-0.03869174841554122)-- (-1.5086427793816215,1.7050853839497757);
\draw [line width=1pt] (-1.5086427793816215,1.7050853839497757)-- (0.6672834888153125,0.8816838281250444);
\draw [line width=1pt] (-1.440669949915849,0.4637856244942844)-- (-0.15769908730316703,1.1938690095039803);
\draw [line width=1pt] (-1.0817149356921705,1.1218371565399394)-- (-0.8866515866329271,0.779053580930514);
\draw [line width=1pt] (-1.0817149356921705,1.1218371565399394) circle (0.788797287cm);
\begin{scriptsize}
\draw[color=black] (-0.6897930376063115,-0.75) node {$\mathcal E$};
\draw[color=black] (-0.85,-1.12) node {$\Gamma$};
\draw [fill=wrwrwr] (-0.8866515866329271,0.779053580930514) circle (1pt);
\draw[color=wrwrwr] (-0.8485707688858103,0.65) node {$Q$};
\draw[color=black] (-0.42,1.8) node {$T_{P_0}\Gamma$};
\draw[color=black] (-2,0.89) node {$T_{P_1}\Gamma$};
\draw[color=wrwrwr] (-1.15,1.23) node {$R$};
\draw [fill=rvwvcq] (-1.440669949915849,0.4637856244942844) circle (1pt);
\draw[color=wrwrwr] (-1.57,0.5) node {$P_0$};
\draw [fill=wrwrwr] (-0.15769908730316703,1.1938690095039803) circle (1pt);
\draw[color=wrwrwr] (-0.12272971160810187,1.2916510015726108) node {$P_1$};
\draw[color=wrwrwr] (-1.5,1.8) node {$P$};
\end{scriptsize}
\end{tikzpicture}
    \caption{The point $P_0\in\Gamma$ is reflected to the point $P_1\in\Gamma$ by the outer-length billiard dynamics around $\mathcal E$.}
    \label{fig:outer_length_ellipse}
\end{figure}
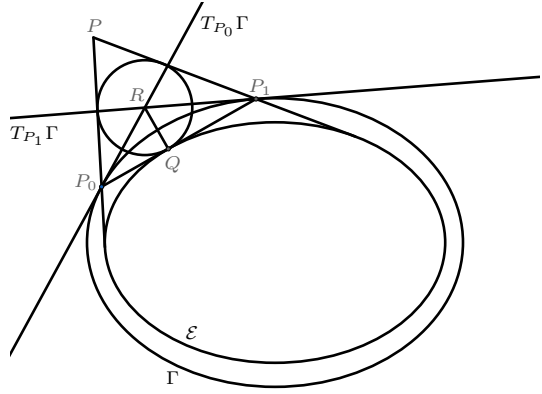

\indent We underline that it would be interesting to investigate if these are the unique cases; this fundamental problem (possibly to be studied by an integral inequality \`a la Bialy \cite{BIALY92}) may present non-trivial difficulties, due to the infinite total area of the phase-space.

\section{Asymptotic expansions}
S. Marvizi and R. Melrose’s theory, first stated and proved for Birkhoff billiards \cite{MaMe}[Theorem 3.2], can be applied to the general case of (strongly) billiard-like maps, see \cite{Glu}[Section 2.1].  As an outcome, the following expansion at $\rho = 0$ of the corresponding minimal average function holds:
$$\beta (\rho) \sim \beta_1 \rho + \beta_3 \rho^3 + \beta_5 \rho^5  + \ldots$$
in terms of odd powers of $\rho$. It is well-known –see e.g. \cite{MaMe}[Section 7] again– that for usual billiards the sequence $\{\beta_k\}$ can be interpreted as a spectrum of a differential operator, see also Remark 2.11 in \cite{ATSym}. The question is widely open for other types of billiards, included outer length billiards. \\
\indent In this section, we gather all the technical results in order to prove the next theorem, providing coefficient $\beta_5$ for the outer length billiard map. This result is a refinement of \cite{MCLure-Vitale}[Theorem 1, point $(iii)$]. In fact, in a genuine framework of convex planar geometry, D.E. Vitale and  R.A. McClure computed $\beta_3$ by using as coordinate the support function and as parameter the angle with respect to a fixed direction.
\begin{theorem} \label{MAIN THEO}
Let $\Omega$ be a strictly-convex planar domain with smooth boundary $\partial \Omega$. Suppose that $\partial \Omega$ has everywhere positive curvature. Denote by $k(s)$ the (ordinary) curvature of $\partial \Omega$ with arc-length parameter $s$. Let $\ell$ be
the length of the boundary and
$$L := \int^\ell_0 k^{2/3}(s) ds.$$
The formal Taylor expansion at $\rho = 0$ of Mather’s $\beta$-function for the outer length billiard map has
coefficients:
\begin{eqnarray*}
\beta_{2k} &=& 0 \text{ for all } k \\
\beta_1 &=& \ell \\
\beta_3 &=& \frac{L^3}{12} \\
\beta_5 &=& L^4\int^\ell_0\left(\frac{k^{4 / 3}(s)}{120}+\frac{k^{-\frac{8}{3}}(s) k^{\prime 2}(s)}{2160} \right)ds.
\end{eqnarray*}
\end{theorem}
\noindent As expected, a straightforward consequence of the previous result is that --as for other billiards-- also for outer length ones, the two coefficients $\beta_1$ and $\beta_3$ allow to recognize a circle among all strictly-convex planar domains.
\begin{corollary}
The coefficients $\beta_1$ and $\beta_3$ recognize a circle. In
particular:
$$3\beta_3+\pi^2\beta_1\leq 0$$
with equality if and only if $\partial \Omega$ is a circle.
\end{corollary}
\begin{proof}
We apply H\"older's inequality with $p=3/2$ and $q=3$ to obtain
\begin{equation}
\label{inequality:minkowski}
L = \int_0^{\ell}k^{2/3}(s)ds\leq\left(\int_0^{\ell}(k^{2/3}(s))^{3/2}ds\right)^{2/3}\left(\int_0^{\ell}1^{3}ds\right)^{1/3} = (2\pi)^{2/3}\ell^{1/3}
\end{equation}
since $\int_0^{\ell}k(s)ds=2\pi$. Using the expressions of $\beta_1$ and $\beta_3$ found in Theorem \ref{MAIN THEO}, we can write
\[
3\beta_3+\pi^2\beta_1=\frac{1}{4}L^3-\pi^2\ell\leq \frac{1}{4}(2\pi)^2\ell-\pi^2\ell = 0.
\]
In the case of equality, namely if $3\beta_3+\pi^2\beta_1$, then $L= (2\pi)^{2/3}\ell^{1/3}$, and the case of equality is reached in \eqref{inequality:minkowski}. In that case, $k$ is constant. Hence $\Omega$ is a disk.
\end{proof}
\begin{remark}
Let $\mathcal{P}_n^c$ be the set of all convex polygons with at most $n$ vertices which are circumscribed to $\Omega$. We define
$$\delta(\Omega;\mathcal{P}_n^c) := \inf \{ \ell(P_n): \ P_n \in \mathcal{P}_n^c\},$$
where $\ell(P_n)$ is the perimeter length of $P_n$. Clearly, essentially in view of equality (\ref{beta-MA}), Theorem \ref{MAIN THEO} gives also the formal expansion of $\delta(\Omega;\mathcal{P}_n^c)$ at $n \to +\infty$. 
\end{remark}
\noindent Since we use the arc-length parametrization of $\partial \Omega$, it is useful to recall that
\begin{equation} \label{derivate gamma}
\begin{cases}
\gamma'' = k J \gamma', \quad \gamma''' = -k^2 \gamma' + k' J \gamma' \\
\gamma^{(4)} = -3kk'\gamma' +(-k^3 + k'') J\gamma' \\
\gamma^{(5)} = (k^4 -4kk'' -3k'^2) \gamma' +(-6k^2k' + k''') J\gamma' \\
\gamma^{(6)} = (10k^3k' - 10k'k'' - 5 kk''') \gamma' + (k^5 - 10 k^2k'' - 15 k k'^2 + k^{(4)}) J \gamma'
\end{cases}
\end{equation}
where $J$ is the rotation of angle $\pi/2$ in the positive verse.
\begin{proposition}
For $0 \le r \le s \le \ell$, it holds
\begin{equation} \label{Taylorgen}
\begin{aligned} 
& H(r,s) =
 (s-r) + \frac{k^2(r)}{12} (s-r)^3 + \frac{k(r)k'(r)}{12} (s-r)^4 + \\
& + \frac{2k^4(r) + 4k'^2(r) + 7k(r)k''(r)}{240} (s-r)^5 + O((s-r)^6)
\end{aligned}
\end{equation}
uniformly as $(s-r) \to 0$.
\end{proposition}
\begin{proof}
We start by writing separately the Taylor expansions of numerator and denominator of the generating function
\begin{equation}\label{genfunc}
H(r,s) = \frac{(\gamma(s) - \gamma(r)) \wedge (\gamma'(s) -\gamma'(r))}{\gamma'(r) \wedge \gamma'(s)}.    
\end{equation}

From now on, we omit the dependence on $r$ of $\gamma$, $k$ and their derivates; moreover, we indicate $\delta := (s-r)$. The Taylor expansion of the numerator is 
\begin{eqnarray*}
& (\gamma(s) - \gamma(r)) \wedge (\gamma'(s) -\gamma'(r)) = \\
& = \left( \gamma'\delta + \frac{\gamma''}{2} \delta^2 + \frac{\gamma'''}{6} \delta^3 + \frac{\gamma^{(4)}}{24} \delta^4 + \frac{\gamma^{(5)}}{5!} \delta^5 + O(\delta^6) \right) \wedge \left( \gamma''\delta + \frac{\gamma'''}{2} \delta^2 + \frac{\gamma^{(4)}}{6} \delta^3 + \frac{\gamma^{(5)}}{24} \delta^4 + \frac{\gamma^{(6)}}{5!} \delta^5 + O(\delta^6) \right) = \\
& = k \delta^2 + \frac{k'}{2}\delta^3 + \frac{1}{6} \left( \frac{2k'' - k^3}{2} \right) \delta^4 + \left( \frac{k''' - 3 k^2k'}{24} \right) \delta^5 + \left(\frac{2k^5 -48 k k'^2 -29 k^2 k'' +6k^{(4)}}{720} \right) \delta^6 + O(\delta^7),
\end{eqnarray*}
where --in the last equality-- we have used formulas (\ref{derivate gamma}). \\
Similarly, the Taylor expansion of the denominator is
\begin{eqnarray*}
& \gamma'(r) \wedge \gamma'(s) = \gamma' \wedge \left( \gamma' + \gamma'' \delta + \frac{\gamma'''}{2} \delta^2 + \frac{\gamma^{(4)}}{6} \delta^3 + \frac{\gamma^{(5)}}{24} \delta^4 + \frac{\gamma^{(6)}}{5!} \delta^5 + O(\delta^6) \right) = \\
& = k \delta + \frac{k'}{2} \delta ^2 + \left(\frac{-k^3 + k''}{6}\right)\delta^3 + \left(\frac{- 6k^2k' + k''' }{24} \right) \delta^4 + \left(\frac{k^5 -10 k^2 k'' -15 k k'^2 + k^{(4)}}{5!}\right) \delta^5 + O(\delta^6) = \\
%& k\delta \left[ 1 - \frac{k'}{2k}\delta + \left( \left( \frac{k'}{2k} \right)^2 + \frac{k^3 - k''}{6 k} \right) \delta^2 + \left( - \left( \frac{k'}{2k} \right)^3 + \frac{k'(-k^3+k'')}{6k^2} + \frac{6k^2k' - k'''}{24k} \right) \delta^3 + A \delta ^4 + O(\delta^5) \right]^{-1} = \\
& = k\delta \left[ 1 - \frac{k'}{2k}\delta + \left( \frac{2k^4 + 3k'^2 - 2kk''}{12 k^2}\right) \delta^2 - \left( \frac{3k'^3 - 2k' (k^4 + 2kk'') + k^2 k'''}{24 k^3} \right) \delta^3 + D \delta ^4 + O(\delta^5) \right]^{-1}
\end{eqnarray*}
where %$$A = - \frac{k^5 - 10 k^2k'' - 15 k k'^2 + k^{(4)}}{5! k} + \left( \frac{k^3-k''}{6k} \right)^2 - \frac{(6k^2k'-k''')k'}{24k^2} + 3 \frac{k'^2(k^3-k'')}{24k^3} + \left(\frac{k'}{2k}\right)^4.$$
$$D = \frac{45 k'^4 - 90 k k'^2 k'' + 30 k^2 k' k''' + 2k^2 (7k^6 + 10 k^3 k'' + 10 k''^2 - 3 k k^{(4)})}{720 k^4}.$$
By using the above expansions of numerator and denominator, we obtain that 
\begin{eqnarray*}
& H(r,s) = %\frac{1}{k \delta} \left( k \delta^2 + \frac{k^3}{12} \delta^4 + \frac{k^2 k'}{12} \delta^5 + O(\delta^6) \right) = \\ & 
\delta + \frac{k^2}{12} \delta^3 + \frac{k k'}{12} \delta^4 + \frac{2k^4 + 4k'^2 + 7kk''}{240} \delta^5 + O(\delta^6),
\end{eqnarray*}
which is the desired result. 
\end{proof}

\begin{proposition} \label{exp} The outer length billiard map 
$T:(s_0,\varepsilon_0)\mapsto(s_{1},\varepsilon_{1})$ has the following expansion:
\begin{equation}
\begin{cases} \label{eccola}
s_{1} = s_0 + \varepsilon_0 \\
\varepsilon_{1} = \varepsilon_0 + A(s_0) \varepsilon_0^2 + B(s_0) \varepsilon_0^3 + C(s_0)\varepsilon_0^4 + O(\varepsilon_0^5)
\end{cases}
\end{equation}
where
$$A(s) = -\frac{2k'(s)}{3k(s)}, \quad B(s) = \frac{10k'^2(s)}{9k^2(s)} - \frac{2 k''(s)}{3k(s)}$$
and
$$C(s) = \frac{-24k^4(s)k'(s) - 1160 k'^3(s) + 1200 k(s) k'(s) k''(s) - 216 k^2(s) k'''(s)}{540 k^3(s)}.$$
\end{proposition}
\begin{proof} We start by writing separately the Taylor expansions of numerator and denominator of the radius $\mathcal{R}$ of the circle in $\mathbb{R}^2 \setminus \Omega$ tangent to $\partial \Omega$ at $\gamma(s_{1})$ and to the line $P\gamma(s_0)$, see Figure \ref{fig:outer_length_billiard_reflection}.
$$\mathcal{R} = \frac{(\gamma(s_1)-\gamma(s_0)) \wedge \gamma'(s_1)}{1+\gamma'(s_1) \cdot \gamma'(s_0)} = \frac{(\gamma(s_2)-\gamma(s_1)) \wedge \gamma'(s_2)}{1+\gamma'(s_2) \cdot \gamma'(s_1)}.$$
From now on, we indicate --by subscripting $1$-- the dependence on $s_1$ of $\gamma$, $k$ and their derivates. Moreover, we recall that $\varepsilon_1 = s_2 - s_1$. The Taylor expansion of the numerator is
\begin{eqnarray*}
& (\gamma(s_2)-\gamma(s_1)) \wedge \gamma'(s_2) = \\
& = \left(\gamma_1' \varepsilon_1 + \frac{\gamma_1''}{2} \varepsilon_1^2 + \frac{\gamma_1'''}{6} \varepsilon_1^3 + \frac{\gamma^{(4)}_1}{24}  \varepsilon_1^4 + \frac{\gamma^{(5)}_1}{5!}  \varepsilon_1^5 + O(\varepsilon_1^6\right) \wedge \left(\gamma'_1 + \gamma_1'' \varepsilon_1 + \frac{\gamma_1'''}{2} \varepsilon_1^2 + \frac{\gamma^{(4)}_1}{6}  \varepsilon_1^3 + \frac{\gamma^{(5)}_1}{24}  \varepsilon_1^4 + O(\varepsilon_1^5)\right) = \\
& = \frac{k_1}{2} \varepsilon_1^2 + \frac{k_1'}{3} \varepsilon_1^3 + \left( \frac{-k_1^3 + 3k_1''}{24}\right) \varepsilon_1^4 + \left( \frac{-9 k_1^2 k_1' + 4 k_1'''}{120}\right) \varepsilon_1^5 + O(\varepsilon_1^6),
\end{eqnarray*} 
where --in the last equality-- we have used formulas (\ref{derivate gamma}). \\
Similarly, the Taylor expansion of the denominator is
\begin{eqnarray*}
& 1+\gamma'(s_2) \cdot \gamma'(s_1) = 1 + \left(\gamma'_1 + \gamma_1'' \varepsilon_1 + \frac{\gamma_1'''}{2} \varepsilon_1^2 + \frac{\gamma^{(4)}_1}{6}  \varepsilon_1^3 + \frac{\gamma^{(5)}_1}{24}  \varepsilon_1^4 + O(\varepsilon_1^5)\right) \cdot \gamma_1' = \\
& = 2 \left( 1 - \frac{k_1^2}{4} \varepsilon_1^2 - \frac{k_1 k_1'}{4} \varepsilon_1^3 + 
%\left( \frac{k_1^4 - 4 k_1k_1'' - 3k_1'^2}{48}\right) \varepsilon_1^4 + 
O(\varepsilon_1^4) \right) = 2 \left( 1 + \frac{k_1^2}{4} \varepsilon_1^2 +  \frac{k_1 k_1'}{4} \varepsilon_1^3 + 
%\left( \frac{2k_1^4 + 4k_1k_1'' + 3k_1'^2}{48} \right) \varepsilon_1^4 +
O(\varepsilon_1^4) \right)^{-1}.
\end{eqnarray*}
By using the above expansions of numerator and denominator, we obtain that 
\begin{equation} \label{R con +}
2 \mathcal{R} = \frac{k_1}{2} \varepsilon_1^2 + \frac{k_1'}{3} \varepsilon_1^3  + \left( \frac{2k_1^3 + 3k_1''}{24} \right) \varepsilon_1^4 + \left( \frac{16k_1^2k_1' + 4 k_1'''}{120}\right) \varepsilon_1^5  + O(\varepsilon_1^6)
\end{equation}
or, equivalently:
\begin{equation} \label{R con -}
2 \mathcal{R} = \frac{k_1}{2} \varepsilon_0^2 - \frac{k_1'}{3} \varepsilon_0^3  + \underbrace{\left( \frac{2k_1^3 + 3k_1''}{24} \right)}_{:= A_4} \varepsilon_0^4 - \underbrace{\left( \frac{16k_1^2k_1' + 4 k_1'''}{120}\right)}_{:= A_5} \varepsilon_0^5  + O(\varepsilon_0^6).
\end{equation}
Substituting the powers of the expansion 
$$\varepsilon_1 = \varepsilon_0 + \alpha(s_1) \varepsilon_0^2 + \beta(s_1) \varepsilon_0^3 + \gamma(s_1)\varepsilon_0^4 + O(\varepsilon_0^5)$$ 
in (\ref{R con +}), we obtain (we omit in the sequel the dependence on $s_0$ in $\alpha, \beta$ and $\gamma$):
\begin{eqnarray*} \label{R con espansione}
2 \mathcal{R} & = & \frac{k_1}{2} \varepsilon_0^2 + \left( k_1 \alpha + \frac{k_1'}{3} \right) \varepsilon_0^3  + \left( \frac{k_1}{2} (\alpha^2 + 2\beta) + k_1' \alpha + A_4 \right) \varepsilon_0^4 + \\
&+& \left( k_1(\alpha\beta + \gamma) + k_1'(\alpha^2 + \beta) + 4 \alpha A_4 + A_5 \right) \varepsilon_0^5  + O(\varepsilon_0^6).
\end{eqnarray*}
By equaling the above expansion to (\ref{R con -}), we have:
$$\alpha(s) = -\frac{2k'(s)}{3k(s)}, \quad \beta(s) = \frac{4k'^2(s)}{9 k^2(s)}$$
and
$$\gamma(s) = \frac{-320 k'^3(s) + 3 k'(s)(-8k^4(s) + 60k(s)k''(s))  - 36 k^2(s)k'''(s)}{540 k^3(s)}.$$ 
Finally, from
$$\varepsilon_1 = \varepsilon_0 + \alpha(s_0) \varepsilon_0^2 + (\alpha'(s_0) + \beta(s_0)) \varepsilon_0^3 + \left( \frac{\alpha''(s_0)}{2} + \beta'(s_0) + \gamma(s_0) \right) \varepsilon_0^4 + O(\varepsilon_0^5),$$ 
we obtain the desired result, that is:
$$A(s) = -\frac{2k'(s)}{3k(s)}, \quad B(s) = \frac{10k'^2(s)}{9k^2(s)} - \frac{2 k''(s)}{3k(s)}$$
and
$$C(s) = \frac{-24k^4(s)k'(s) - 1160 k'^3(s) + 1200 k(s) k'(s) k''(s) - 216 k^2(s) k'''(s)}{540 k^3(s)}.$$
\end{proof}
\begin{proposition}\label{Taylorsk}
Let $q \ge 3$. The $q$-periodic orbits of rotation number $1/q$ for the outer length billiard map 
%$T:(s_k,\varepsilon_k)\mapsto(s_{k+1},\varepsilon_{k+1})$ 
have the following expansion:
\begin{equation}
\begin{cases} \label{eccola bis}
s_{k} = s_0^q + a_0\left(k/q\right) + \frac{a_1\left(k/q\right)}{q} + \frac{a_2\left(k/q\right)}{q^2} + O\left(\frac{1}{q^3}\right) \\ \\
\varepsilon_{k} = \frac{b_1\left(k/q\right)}{q} + \frac{b_2\left(k/q\right)}{q^2} + \frac{b_3\left(k/q\right)}{q^3} + O\left(\frac{1}{q^4}\right)
\end{cases}
\end{equation}
where $s_0^q\in\RR$ converges to $0$ with $q$, $a_0:\RR \to \RR$ is a map such that $a_0(x+1)=a_0(x)+\ell$ for any $x$ and
$a_1,a_2,b_1,b_2,b_3:\RR\to\RR$ are $1$-periodic maps which can be expressed as
%\luca{ $s(x)$ is $a_0(x)$, should we replace $s(x)$ inside the other functions ?}
\begin{equation}\label{conto diretto}
\begin{cases}
a_0^{-1}(s) = \frac{1}{L} \int^s_0 k^{2/3}(r) dr := x(s), \qquad L := \int^\ell_0 k^{2/3}(r) dr \\ \\
a_1(x) = 0 \\ \\
a_2(x)=k^{-\frac{2}{3}}(a_0(x))\left(\int^x_0 L^3\left(\frac{1}{810} \left( 9{k^{\prime\prime}}{k^{-\frac{7}{3}}} -12{(k')^2}{k^{-\frac{10}{3}}}\right)+\frac{k^{\frac{2}{3}}}{15}\right)(a_0(t)) dt +cx\right) \\ \\
b_1(x) = a_0'(x) = L k^{-2/3}(a_0(x))\\ \\
b_2(x) = \frac{a_0''(x)}{2} = -\frac{L^2k'(a_0(x)) k^{-7/3} (a_0(x))}{3} \\ \\
b_3(x) = a_2' + \frac{a_0'''}{6}.
\end{cases}
\end{equation}
The constant $c$ in the expression of $a_2$ is such that $L^3\left(\frac{1}{810} \left( 9k^{\prime\prime}{k^{-\frac{7}{3}}} -12{(k')^2}{k^{-\frac{10}{3}}}\right)+\frac{k^{\frac{2}{3}}}{15}\right) +c$ has zero mean.
\end{proposition}
\begin{proof} 
Since the points in the orbits are equidistributed as $q\to+\infty$, for any $q$ we can choose the first point of the orbit $s_0^q$ such as $s_0^q\to 0$ for $q \to +\infty$. To simplify the notations, we omit the dependence of $a_i$ and $b_j$ on $k/q$. \\
On one hand, combining the expansions in (\ref{eccola}), we have
\begin{eqnarray*}
& \varepsilon_{k+1} - \varepsilon_k = A(s_k) \varepsilon_k^2 + B(s_k) \varepsilon_k^3 + C(s_k)\varepsilon_k^4 + O(\varepsilon_k^5) = \\
& = \frac{A(s_0^q+a_0) b^2_1}{q^2} + \frac{B(s_0^q+a_0)b_1^3 + A'(s_0^q+a_0) a_1b_1^2 + 2A(s_0^q+a_0) b_1 b_2}{q^3} + \frac{F(a_i,b_j)}{q^4} + O\left( \frac{1}{q^5} \right),
\end{eqnarray*}
where
%\corentin{Maybe I remove this, since we don't need the exact expression of $a_1$?}
\begin{eqnarray*}
F(a_i,b_j) & := & A(s_0^q+a_0)b_2^2 + 2A(s_0^q+a_0) b_1b_3 + 2A'(s_0^q+a_0) a_1b_1b_2 + A'(s_0^q+a_0) a_2b_1^2 + \\
& + & A''(s_0^q+a_0) a_1^2b_1^2/2 + 3B(s_0^q+a_0)b_1^2b_2 + B'(s_0^q+a_0) a_1b_1^3 + C(s_0^q+a_0) b_1^4.
\end{eqnarray*}
Moreover, directly from the second expansion in (\ref{eccola bis}), we have
\begin{eqnarray*}
& \varepsilon_{k+1} - \varepsilon_k = \frac{b_1'}{q^2} + \frac{b_2'+b_1''/2}{q^4} + \frac{b_3' + b_2''/2 + b_1'''/6}{q^5} + O\left( \frac{1}{q^5} \right).
\end{eqnarray*}
Equaling these two expansions, we obtain that $a_i$ and $b_j$ solve
\begin{equation} \label{prima}
\begin{cases}
A(s_0^q+a_0) b^2_1 = b_1' \\
B(s_0^q+a_0)b_1^3 + A'(s_0^q+a_0) a_1b_1^2 + 2A(s_0^q+a_0) b_1 b_2 = b_2' + b_1''/2 \\
F(a_i,b_j) = b_3' + b_2''/2 + b_1'''/6
\end{cases}
\end{equation}
From the other hand, directly from the first expansion in (\ref{eccola}), we conclude that
\begin{eqnarray*}
& s_{k+1} - s_k = \frac{a_0'}{q} + \frac{a_1' + a_0''/2}{q^2} + \frac{a_2' + a_1''/2 + a_0'''/6}{q^3} + O\left( \frac{1}{q^4} \right),
\end{eqnarray*}
which --compared which the second expansion in (\ref{eccola})-- gives the system
\begin{equation} \label{seconda}
\begin{cases}
a_0' = b_1 \\
a_1' + a_0''/2 = b_2 \\
a_2' + a_1''/2 + a_0'''/6 = b_3
\end{cases}
\end{equation}
%\vspace{0.20cm}
\noindent $\bullet$ Expressions of $a_0$ and $b_1$.
To compute $a_0$ and $b_1$, we solve the system
\begin{equation} \label{system:a_0_b_1}
\begin{cases}
b_1 = a_0' \\
b_1' = A(s_0^q+a_0) b^2_1
\end{cases}
\end{equation}
Replacing $b_1$ by $a_0'$ in the second equation, it gives
\begin{equation}
    \label{equation:solving_a_0}
a_0''=(a_0')^2A(s_0^q+a_0).
\end{equation}
If we denote by $A_1(s) = -\tfrac{2}{3}\log k(s)$ a primitive of $A$, then from Equation \eqref{equation:solving_a_0} follows
\[
\left(a_0'e^{-A_1(s_0^q+a_0)}\right)'=0.
\]
Hence $a_0'e^{-A_1(s_0^q+a_0)}$ is constant. Consider now $A_2(s) = \int_0^sk^{2/3}(r)dr$, which is a primitive of $\exp(-A_1)$. We just proved that $A_2(s_0^q+a_0)$ has constant derivative, hence it must be of the form $A_2(s_0^q + a_0(x)) = ux+v$ for any $x\in\RR$, where $u,v\in\RR$. Since, by definition, $A_2(s_0^q + a_0(0)) = A_2(s_0^q) = v$, we have $v = A_2(s_0^q)$. The expression of $u$ is given by $u = A_2(s_0^q + a_0(1)) - A_2(s_0^q) = A_2(s_0^q + \ell) - A_2(s_0^q) = \int^{\ell}_0 k^{2/3}(r) dr$. Finally, $b_1$ follows from $b_1 = a_0'$. \\
\noindent $\bullet \bullet$ Expressions of $a_1$ and $b_2$. To compute $a_1$ and $b_2$, we solve the system
\begin{equation} \label{system:a_1_b_2}
\begin{cases}
b_2 = a_1' + a_0''/2\\
b_2' + b_1''/2 = B(s_0^q + a_0)b_1^3 + A'(s_0^q + a_0) a_1b_1^2 + 2A(s_0^q + a_0) b_1 b_2.
\end{cases}
\end{equation}
Note that the terms note containing $a_1$ nor $b_2$ can be computed using the expression of $a_0$ and $b_1$ we just obtained. Let us replace in the second equation of \eqref{system:a_1_b_2} $b_2$ by the expression given by the first equation: we obtain an equation for which we split the terms containing $a_1$ from the others. Namely,
\begin{equation}
    \label{equation:solving_a_1}
a_1''-2A(s_0^q + a_0)b_1a_1'-A'(s_0^q + a_0)a_1b_1^2 
= 
A(s_0^q + a_0)b_1a_0''+B(s_0^q + a_0)b_1^3-\tfrac{1}{2}b_1''-\tfrac{1}{2}a_0^{(3)}.
\end{equation}
Replacing $a_0$ and $b_1$ by the expressions we just found, the left-hand side of \eqref{equation:solving_a_1} can be expressed as
\[
a_1''+\tfrac{4L}{3}k^{-5/3}k'a_1'+\tfrac{2L^2}{3}\left(k^{-7/3}k''-k^{-10/3}{k'}^2\right)a_1 = k^{-2/3}(a_1k^{2/3})'',
\]
where it is implicitly understood that $k$ and its derivatives are evaluated in $s_0^q + a_0$. The right-hand side of \eqref{equation:solving_a_1} vanishes. Hence equation \eqref{equation:solving_a_1} is equivalent to
\[
k^{-2/3}(a_1k^{2/3})''=0.
\]
Since $a_1$ is periodic and vanishes at $0$, we necessarily have $a_1=0$. The expression of $b_2$ comes from the first equation of \eqref{system:a_1_b_2}, namely $b_2=a_0''/2$. \\
%\vspace{0.20cm}
$\bullet \bullet \bullet$ Expressions of $a_2$ and $b_3$. Although it will not be employed in the subsequent computations, we shall derive an explicit expression for the coefficient $a_2$. By making use of equations \eqref{prima} and \eqref{seconda}, and taking into account that $a_1 = 0$, we obtain the following system:
\begin{equation}\label{terza}
    %\begin{equation}
\begin{cases}
b_3=a_2^{\prime}+\frac{a_0'''}{6} \\
A^{\prime}(s_0^q + a_0) b_1^2 a_2+A(s_0^q + a_0)\left(b_2^2+2 b_1 b_3\right)+ 
B(s_0^q + a_0) 3 b_1^2 b_2+C(s_0^q + a_0) b_1^4=\frac{b_1^{\prime \prime\prime}}{6} +\frac{b_2^{\prime\prime}}{2} +b_3^{\prime}
\end{cases}
%\end{aligned}
\end{equation}
From the first equation of \eqref{terza} we have
$ b_3^\prime=a_2^{\prime\prime} +\frac{a_0^{(4)}}{6} $ which in turn gives
$$ b_3^\prime=a_2^{\prime\prime} +\left( \frac{11 k^\prime k^{\prime\prime}k^{-\frac{10}{3}}}{27} -\frac{8 (k^\prime)^3 k^{\frac{13}{3}}}{27}  -\frac{k^{\prime\prime\prime}k^{-\frac{7}{3}}}{9}\right).$$
Replacing into the second of \eqref{terza} and grouping all the terms with $a_2$, we get 
\begin{equation}
    (k^{\frac{2}{3} }a_2)^{\prime\prime}  =L^4\left(\frac{40(k^\prime)^3-45kk^\prime k^{\prime\prime}+9k^2k^{\prime\prime}}{810k^5} +\frac{2 k^\prime}{45 k}\right).
\end{equation}
The right-hand side is the derivative of 
$$L^3\left(\frac{1}{810} \left( 9k^{\prime\prime}{k^{-\frac{7}{3}}} -12{(k')^2}{k^{-\frac{10}{3}}}\right)+\frac{k^{\frac{2}{3}}}{15}\right) +c,$$
where $c$ is a constant such that this function has zero mean. Consequently, at this point we can integrate another time and get
$$ a_2(x)=k^{-\frac{2}{3}}(a_0(x))\left(\int^x_0 L^3\left(\frac{1}{810} \left( 9{k^{\prime\prime}}{k^{-\frac{7}{3}}} -12{(k')^2}{k^{-\frac{10}{3}}}\right)+\frac{k^{\frac{2}{3}}}{15}\right)(a_0(t)) dt +cx\right) .$$
The value of $b_3$ can now be easily derived from the first one of (\ref{terza}).
%\noindent -- Expressions of $a_2$ and $b_3$.
%Note that we don't express $a_2$ explicitely -- it will not be useful for what follows, and moreover its expression is rather complicated. The expression of $b_3$ depending on $a_2$ and $a_0$ comes from \eqref{seconda}, together with the previous result stating that $a_1=0$.
\end{proof}

\section{Proof of Theorem \ref{MAIN THEO}}
This section is entirely devoted to the proof of Theorem \ref{MAIN THEO}, providing coefficient $\beta_5$ for the outer length billiard map.
\begin{proof}
    We start the computation of the beta function by writing its value at rational points of the form \(\frac{1}{q}\), which (by the expansion \eqref{Taylorgen} of the generating function $H$) %and \eqref{eccola bis}-- 
    is
    \begin{equation}\label{Taylorbeta}
        \beta \left(\frac 1q \right)=\frac{1}{q}\sum_{n=0}^{q-1} H\left(s_n, s_{n+1}\right)=\frac{1}{q}\sum_{n=0}^{q-1} \varepsilon_n+\frac{k^2}{12} \varepsilon_n^3+\frac{k k'}{12}  \varepsilon_n^4 + \frac{2k^4 + 4k'^2 + 7kk''}{240}  \varepsilon_n^5 +O(\varepsilon_n^6).
    \end{equation}
Here, the curvature $k$ and its derivates $k'$ and $k''$ are to be understood as evaluated in \( s_n \). \\
Now, we substitute in the above formula \( s_n \) and \( \varepsilon_n \) with their corresponding Taylor expansions obtained in Proposition \ref{Taylorsk}. We then proceed to group the various terms according to their order of magnitude \( q^k \). \\
~\newline 
First, we observe that the summation of \( \varepsilon_n \) is simply equal to the perimeter \( \ell \) of \( D \), so that $\beta_1 = \ell$. \\
~\newline By inspecting the formula even before performing the substitution, we see that there are no terms of order \( q^{-2} \), so that $\beta_2 = 0$, as expected by Marvizi-Melrose’s theory. \\
~\newline 
The second term of the summation on the right-hand side of \eqref{Taylorbeta} becomes, after the substitution and after grouping the various powers of $q$,
\begin{equation}\label{contrib1}
\begin{aligned}
& \frac{1}{12} \sum_{n=0}^{q-1}  k^2\left(s_n\right) \varepsilon_n^3 = \frac{1}{12} \sum_{n=0}^{q-1}  k^2\left(a_0+\frac{a_2}{q^2}+O\left(\frac{1}{q^3}\right)\right)\left(\frac{b_1}{q}+\frac{b_2}{q^2}+\frac{b_3}{q^3}\right)^3= \\
& = \sum_{n=0}^{q-1} \frac{k^2 b_1^3}{12} \frac{1}{q^3}+\frac{k^2 b_1^2 b_2}{4} \frac{1}{q^4}+\frac{2 k k' b_1^3 a_2+3 k^2\left(b_1^2 b_3+b_1 b_2^2\right)}{12} \frac{1}{q^5}+O\left(\frac{1}{q^6}\right).
\end{aligned}
\end{equation}
Similarly, we have that
\begin{equation} \label{contr2}
    \begin{aligned}
& \frac{1}{12} \sum_{n=0}^{q-1}  k\left(s_n\right) k^{\prime}\left(s_n\right) \varepsilon_n^4 = \frac{1}{12} \sum_{n=0}^{q-1}  k k^{\prime} %\left(a_0+\frac{a_2}{q^2}+O\left(\frac{1}{q^3}\right)\right)
\left(\frac{b_1}{q}+\frac{b_2}{q^2}+O\left(\frac{1}{q^3}\right)\right)^4= \\
& = \sum_{n=0}^{q-1} \frac{k k^{\prime} b_1^4}{12}  \frac{1}{q^4} + \frac{k k^{\prime}b_1^3 b_2}{3}  \frac{1}{q^5}+O\left(\frac{1}{q^6}\right).
\end{aligned}
\end{equation}
Finally, the last term is
\begin{equation}\label{contr3}
    \sum_{n=0}^{q-1} \left( \frac{2 k^4+4 k'^2+7 k k''}{240} \right) b_1^5\frac{1}{q^5} +O\left(\frac 1 {q^6}\right).
\end{equation}
We recall that, in the last three formulas, it is implicitly understood that all functions \( a_i, b_i \) are evaluated at \( n/q \), and the curvature $k$ and its derivatives $k'$ and $k''$, where not explicitly specified, are computed at \( s_0^q + a_0(n/q) \). To determine \( \beta_3 \), we compute
\(\lim_{q\to +\infty} q^3 \left( \beta\left(\frac{1}{q}\right) - \frac{\ell}{q} \right) \). \\
From formulas \eqref{contrib1}, \eqref{contr2}, and \eqref{contr3}, we simply obtain that
\begin{equation*}
        \beta_3 = \lim_{q\to +\infty} \frac{1}{12 q} \sum^{q-1}_{n=0} \left(k^2b_1^3 +O\left(\frac 1q\right)\right)
\end{equation*}
By Proposition \ref{Taylorsk}, we have that $b_1 = L k^{-\frac{2}{3}} \Rightarrow k^2 b_1^3 = L^3$, so that 
$$\beta_3=\frac{1}{12}L^3 = \frac{1}{12} \left(\int^\ell_0 k^{2/3}(r) dr\right)^3.$$
Note that the leading part of this limit is constant, while the term denoted by \( O(1/q) \) contains only higher-order terms. We will take this into account when analyzing \( \beta - \ell /q - \beta_3/q^3 \), considering only the terms present in \( O(1/q) \). \\
~\newline
Regarding the terms of order 4, we obtain the following expression:
\begin{equation*}
     \sum_{n=0}^{q-1}\left( \frac{k^2  b_1^2 b_2}{4} + \frac{k k^{\prime} b_1^4}{12} \right) \frac{1}{q^4}.
\end{equation*}
Since, by Proposition \ref{Taylorsk}, we have that $b_1 = L k^{-\frac{2}{3}}$ and $b_2=-\frac{L^2 k'k^{-\frac{7}{3}}}{3}$, we immediately conclude (again, as expected by Marvizi-Melrose’s theory), that $\beta_4=0$. \\
~\newline
The terms of order $5$ are
\begin{equation*}
   \sum_{n=0}^{q-1} \left[\frac{2 k^4+4 k^{\prime 2}+7 k k^{\prime \prime}}{240} b_1^5+\frac{1}{12}\left(2 k k^{\prime} b_1^3 a_2 + 3k^2\left(b_3 b_1^2 + b_2^2 b_1\right)+4 k k^{\prime} b_1^3 b_2\right)\right]\frac{1}{q^5}.
\end{equation*}
Since --by Proposition \ref{Taylorsk})-- \( b_3 = a_2'+\frac{a_0'''}{6} \), we substitute it into the previous equation and separate the summation containing the terms with \( a_2 \):
\begin{equation*}
    \underbrace{\sum_{n=0}^{q-1} \frac{1}{12}\left(2 k k^{\prime} b_1^3 a_2+3 k^2 b_1^2 a_2^{\prime}\right)}_{:=S_1} + \underbrace{\sum_{n=0}^{q-1} \left(\frac{2 k^4+4 k^{\prime 2}+7 k k^{\prime \prime} }{240} b_1^5+\frac{1}{12}\left(3k^2 b_2^2 b_1+4 k k^{\prime} b_1^3 b_2+\frac{1}{2}k^2 a_0^{\prime \prime\prime} b_1^2\right)\right)}_{:= S_2}.
\end{equation*}
We remark that the sum \( S_1 \) contains \( a_2 \) and the sum \( S_2 \) doesn't contain \( a_2 \). As established earlier, we have $$ \beta_5 = \lim_{q\to +\infty} q^5 \left( \beta\left(\frac{1}{q}\right)-\frac{\ell}{q} -\frac{\beta_3}{q^3} \right)=\lim_{q\to +\infty} \frac{1}{q}(S_1+S_2).$$
By studying the limit $\lim_{q\to +\infty} \frac{1}{q}S_1$, we obtain:
\begin{equation}\label{S1}
\begin{aligned}
    &\frac{1}{12} \lim _{q \rightarrow+\infty}  \frac{1}{q} \sum_{n=0}^{q-1} \left(2 k k^{\prime} k^{-2} a_2L^3+3 k^2 k^{-\frac{4}{3}} a_2^{\prime} L^2\right)=\\ & \frac{1}{12} \int_0^1 \left(\frac{2k^{\prime}}{k}\left(a_0(x)\right) a_2(x) L^3+3 k^{\frac{2}{3}}\left(a_0(x)\right) a_2^{\prime}(x) L^2\right) d x
\end{aligned}
\end{equation}
where --once again-- in the summations we have used the convention that the functions \( a_i, b_i \) are evaluated at $n/q$, while the functions \( k, k' \) are evaluated at \( a_0(n/q) \). Similarly, in the integral on the right-hand side, \( a_i, b_i \) are evaluated at \( x \), and \( k, k' \) at \( a_0(x) \).
By integrating by parts the second term inside the integral, we have
\begin{equation} \label{magia}
    \int_0^1 k^{\frac{2}{3}}\left(a_0(x)\right) a_2^{\prime}(x) L^2d x=\left.k^{\frac{2}{3}}\left(a_0(x)\right) a_2(x) L^2\right|_0 ^1-\int_0^1 \frac{2k^{\prime}}{3k}\left(a_0(x)\right)a_2(x) L^3 d x .
\end{equation}
By periodicity, the first term is $0$. By substituting the remaining expression of (\ref{magia}) inside \eqref{S1}, we conclude that the first limit is $0$. \\
Let us proceed with the calculation of the limit $\lim_{q\to +\infty} \frac{1}{q}S_2$. Taking into account that \( a_0'(x) = L k^{-\frac{2}{3}}(a_0(x)) \), see (\ref{conto diretto}), we have that
\[ a_0^{\prime \prime \prime} = L^3 \left( -\frac{2 k^{\prime \prime}}{3k^3} + \frac{14 k'^2}{9k^4} \right). \]  
Taking into account the expressions of $a_i, b_j$ given in (\ref{conto diretto}) and by substituting the previous expression into \( S_2 \), we obtain:
%\begin{equation}\label{S2finale}
%\begin{aligned} \lim_{q\to+\infty} &\frac 1q S_2=\lim_{q\to+\infty}\frac{1}{q}\sum_{n=0}^{q-1}  \frac{2 k^4+4 k^{\prime 2}+7 k k^{\prime \prime}}{240} b_1^5+\frac{1}{12} k^2\left(\frac{a_0^{\prime \prime \prime} b_1^2}{2}+3 b_2^2 b_1\right)= \\
%& =\lim_{q\to+\infty}\frac{1}{q}\sum_{n=0}^{q-1}\left(\frac{k^{2 / 3}}{120}+\frac{k^{-\frac{10}3} k^{\prime 2}}{60}+\frac{7}{240} k^{-\frac{7}{3}} k^{\prime \prime}-\frac{1}{36} k^{-\frac{7}{3}} k^{\prime \prime}+\frac{7}{36 \cdot 3} k^{-\frac{10}{3}} k^{\prime 2}+\frac{1}{36} k^{-\frac{10}{3}} k^{\prime 2}\right) L^5= \\
%& =\lim_{q\to+\infty}\frac{1}{q}\sum_{n=0}^{q-1}\left(\frac{k^{2 / 3}}{120}+\frac{59}{540} k^{-\frac{10}{3}} k^{\prime 2}+\frac{1}{720} k^{-\frac{7}{3}} k^{\prime \prime}\right) L^5=\\
%&=L^5\int^1_0 \left(\frac{k^{2 / 3}}{120}+\frac{59}{540} k^{-\frac{10}{3}} k^{\prime 2}+\frac{1}{720} k^{-\frac{7}{3}} k^{\prime\prime}\right) dx 
%\end{aligned}
%\end{equation}

\begin{equation}\label{S2finale}
\begin{aligned} & \lim_{q\to+\infty} \frac 1q S_2=\lim_{q\to+\infty}\frac{1}{q}\sum_{n=0}^{q-1}  \frac{2 k^4+4 k^{\prime 2}+7 k k^{\prime \prime}}{240} b_1^5+\frac{1}{12}\left(3k^2 b_2^2 b_1+4 k k^{\prime} b_1^3 b_2+\frac{1}{2}k^2 a_0^{\prime \prime\prime} b_1^2\right) = \\
%& =\lim_{q\to+\infty}\frac{1}{q}\sum_{n=0}^{q-1}\left(\frac{k^{2 / 3}}{120}+\frac{k^{-\frac{10}3} k^{\prime 2}}{60}+\frac{7 k^{-\frac{7}{3}} k^{\prime \prime}}{240} -\frac{k^{-\frac{7}{3}} k^{\prime \prime}}{36} + \frac{7 k^{-\frac{10}{3}} k^{\prime 2}}{36 \cdot 3} +\frac{k^{-\frac{10}{3}} k^{\prime 2}}{36} {\bf-\frac{{k'}^2k^{-10/3}}{9}}\right) L^5= \\
& =\lim_{q\to+\infty}\frac{L^5}{q}\sum_{n=0}^{q-1}\left(\frac{k^{2/3}}{120} -\frac{k^{-\frac{10}{3}} k^{\prime 2}}{540} +\frac{k^{-\frac{7}{3}} k^{\prime \prime}}{720} \right) = L^5\int^1_0 \left(\frac{k^{2 / 3}}{120}-\frac{k^{-\frac{10}{3}} k^{\prime 2}}{540} +\frac{k^{-\frac{7}{3}} k^{\prime\prime}}{720} \right) dx.
\end{aligned}
\end{equation}
We finally integrate by parts the last term of the integral, obtaining:
\[ \begin{aligned}
& \int_0^1 L^5 k^{-\frac{7}{3}} k^{\prime \prime} d x= L^4 \int_0^1  k^{-5 / 3} k^{\prime \prime}\left(k^{-\frac{2}{3}} L\right)  d x= L^4 \int_0^1  k^{-\frac{5}{3}}\left(k'\right)' d x=   \frac{5 L^5}{3} \int_0^1  k^{-\frac{10}3} k^{\prime 2} d x.
\end{aligned}\]
%{\color{blue}
%Here I get a + since the $-$ of the derivative cancels out with the one of the integration by part:
%\[
%\int_0^1 L^5 k^{-\frac{7}{3}} k^{\prime \prime} d x = \frac{5}{3} \int_0^1 L^5 k^{-\frac{10}3} %k^{\prime 2} d x.
%\]
%}
Replacing the expression above in \eqref{S2finale}, we conclude that
\[
\lim_{q\to+\infty} \frac 1q S_2 = L^5\int^1_0 \left(\frac{k^{2 / 3}}{120}+\frac{k^{-\frac{10}{3}} k^{\prime 2}}{2160}  \right)dx
\]
%\olga{This morning I did all the computations in order to check the coefficients.. Only in the end I found a difference... But perhaps I am cooked.}
Finally, switching to arc length as the variable of integration, we obtain the desired result: %\luca{All coefficients to be checked. }
%\[
%\beta_5=L^4\int^l_0\left(\frac{k^{4 / 3}}{120}+\frac{77}{720} k^{-\frac{8}{3}} k^{\prime 2} \right)ds
%\]
\[
\beta_5=L^4\int^\ell_0\left(\frac{k^{4 / 3}(s)}{120}+\frac{k^{-\frac{8}{3}}(s) k^{\prime 2}(s)}{2160} \right)ds.
\]
\end{proof}

\section{Lazutkin coordinates and caustics}

%Consider a strictly convex bounded domain $\Omega\subset\RR^2$ with smooth boundary. 

%\begin{definition}
%We say that a Jordan curve $\Gamma\subset \RR^2$ containing $\Omega$ in its bounded component is a \textit{caustic} for the outer-length billiard reflection outslide $\Omega$ if the outer-length billiard map outside $\Omega$ sends any point $p_0\in\Gamma$ to a point $p_1\in\Gamma$.
%\end{definition}

A consequence of Proposition \ref{exp} is that we can compute explicitly Lazutkin coordinates \cite{LAZ} for order 4 in the case of outer length billiards.  
\begin{lemma}[Lazutkin for outer length billiards] The coordinates 
\begin{eqnarray*}
x(s) &=& \frac{1}{L} \int^s_0 k^{2/3}(r) dr, \qquad L:= \int^\ell_0 k^{2/3}(r) dr \\
y(s,\varepsilon) &=& x(s + \varepsilon) - x(s)
\end{eqnarray*}
are so that the outer length billiard dynamics is given by
$$x \mapsto x + y, \qquad y \mapsto y + O(y^4).$$
\end{lemma}
\begin{proof} Let 
$$(s,\varepsilon) \mapsto (x,y) := (f(s),f(s+\varepsilon) - f(s))$$ 
a change of coordinates so that $(x_k,y_k) \mapsto (x_k + y_k, y_{k+1})$. Then --by using the expansion of $\varepsilon_1$ given in (\ref{exp})-- we have
$$\begin{aligned}
y_1 & = x_2 - x_1 = f(s_1+\varepsilon_1) - f(s_1) = f'(s_1) \varepsilon_1 + \frac{f''(s_1)}{2} \varepsilon_1^2 + \frac{f'''(s_1)}{6} \varepsilon_1^3 + O(\varepsilon_1^4) = \\
&= \left( f'(s_0) + f''(s_0) \varepsilon_0 + \frac{f'''(s_0)}{2} \varepsilon_0^3 \right) \left( \varepsilon_0 + A(s_0)\varepsilon_0^2 + B(s_0) \varepsilon_0^3 \right) + \\
& + \left( f''(s_0) + f'''(s_0) \varepsilon_0 \right) \frac{\left( \varepsilon_0 + A(s_0)\varepsilon_0^2 + B(s_0) \varepsilon_0^3 \right)^2}{2} + \frac{f'''(s_0)}{6} \varepsilon_0^3 + O(\varepsilon_0^4) = \\
& = \left( f'(s_0) \varepsilon_0 + \frac{f''(s_0)}{2} \varepsilon_0^2 + \frac{f'''(s_0)}{6} \varepsilon_0^3 \right) + \left( f''(s_0) + f'(s_0) A(s_0) \right) \varepsilon_0^2 + \\
& + \left( f'(s_0) B(s_0) + 2 f''(s_0) A(s_0) + f'''(s_0) \right) \varepsilon_0^3 +  O(\varepsilon_0^4) = \\
& = y_0 + \left( f''(s_0) + f'(s_0) A(s_0) \right) \varepsilon_0^2 + \left( f'(s_0) B(s_0) + 2 f''(s_0) A(s_0) + f'''(s_0) \right) \varepsilon_0^3 +  O(\varepsilon_0^4). 
\end{aligned}$$
Consequently, if we want to get rid of the $\varepsilon_0^2$ and $\varepsilon_0^3$ terms, we need to choose $f$ solving
$$
\begin{cases}
f''(s_0) + f'(s_0) A(s_0) = 0 \\
f'(s_0) B(s_0) + 2 f''(s_0) A(s_0) + f'''(s_0) = 0
\end{cases}
$$
Integrating the first equation, we immediately obtain the desired formula for $f$, giving (up to normalization):
$$x(s) = \frac{1}{L} \int^s_0 k^{2/3}(r) dr, \qquad L:= \int^\ell_0 k^{2/3}(r) dr.$$
Then, by direct computation, it is easy to check that such a function solves also the second equation.
\end{proof}
\indent As a consequence, the outer length billiard map is a small perturbation of the integrable map 
$$(x,y)\mapsto(x + y,y),$$ 
satisfying the assumptions of Lazutkin’s theorem \cite{LAZ}[Theorem 1]. Applying this theorem, the next corollary of Proposition \ref{exp} immediately follows. 
\begin{theorem}
Arbitrary close to the boundary $\partial \Omega$, there exist smooth caustics for the outer length billiard map. The union of these caustics has positive measure.
\end{theorem}
%\olga{The non-persistence of convex resonant caustics to be added?}
\indent On the other hand --regarding the non existence of caustics-- we underline that the following outer length billiard version of Mather’s theorem still holds. 
\begin{theorem} If the curvature of the boundary $\partial \Omega$ vanishes at some point, then the outer length billiard in $\partial \Omega$ has no caustics.
\end{theorem}
\begin{proof} We use Mather’s necessary analytic condition for the existence of a caustic \cite{Mather82}, that is
$$H_{22}(s_0, s_1) + H_{11}(s_1, s_2) < 0.$$
By using the general expression of the generating function \eqref{genfunc}, it is easily seen that
$$\begin{aligned}
& H_{1}(s_1,s_2)= -1 - \frac{(\gamma(s_2) - \gamma(s_1)) \wedge \gamma''(s_1)}{\gamma'(s_1) \wedge \gamma'(s_2)} - \\
& \frac{(\gamma(s_2) - \gamma(s_1)) \wedge (\gamma'(s_2) - \gamma'(s_1)) \cdot (\gamma''(s_1) \wedge \gamma'(s_2))}{(\gamma'(s_1) \wedge \gamma'(s_2))^2}.
\end{aligned}$$
Hence, we have that
$$\begin{aligned} & H_{11}(s_1,s_2)=\frac{\gamma^{\prime}\left(s_1\right) \wedge \gamma^{\prime \prime}\left(s_1\right)}{\gamma^{\prime}\left(s_1\right) \wedge \gamma^{\prime}\left(s_2\right)}-\frac{\left(\gamma\left(s_2\right)-\gamma\left(s_1\right)\right) \wedge \gamma^{\prime \prime \prime}\left(s_1\right)}{\gamma^{\prime}\left(s_1\right) \wedge \gamma^{\prime}\left(s_2\right)} + \\ & +2 \frac{\left(\gamma\left(s_2\right)-\gamma\left(s_1\right)\right) \wedge \gamma^{\prime \prime}\left(s_1\right)}{\left(\gamma^{\prime}\left(s_1\right) \wedge \gamma^{\prime}\left(s_2\right)\right)^{2}}\left(\gamma^{\prime\prime}\left(s_1\right) \wedge \gamma^{\prime}\left(s_2\right)\right) + \\ & +\frac{\gamma^{\prime \prime}\left(s_1\right) \wedge \gamma^{\prime}\left(s_2\right)}{\left(\gamma^{\prime}\left(s_1\right) \wedge \gamma^{\prime}\left(s_2\right)\right)^2}-\frac{\left(\gamma\left(s_2\right)-\gamma\left(s_1\right)\right) \wedge\left(\gamma^{\prime}\left(s_2\right)-\gamma^{\prime}\left(s_1\right)\right)}{\left(\gamma^{\prime}\left(s_1\right) \wedge \gamma^{\prime}\left(s_2\right)\right)^2}\left(\gamma^{\prime \prime \prime}\left(s_1\right) \wedge \gamma^{\prime}\left(s_2\right)\right) + \\ & +2 \frac{\left(\gamma\left(s_2\right)-\gamma\left(s_1\right)\right) \wedge\left(\gamma^{\prime}\left(s_2\right)-\gamma^{\prime}\left(s_1\right)\right)}{\left(\gamma^{\prime}\left(s_1\right) \wedge \gamma^{\prime}\left(s_2\right)\right)^3}\left(\gamma^{\prime \prime}\left(s_1\right) \wedge \gamma^{\prime}\left(s_2\right)\right)^2 .\end{aligned}$$
 Let us now assume that at a point on the boundary corresponding to the arc-length parameter value \( s_1 \), the curvature is zero, that is, \( k(s_1) = 0 \). Since the set is convex, this condition implies that also \( k'(s_1) = 0 \). From formulas \eqref{derivate gamma}, it follows that \( \gamma''(s_1) = \gamma'''(s_1) = 0 \). Substituting into the previous formula, we see that all the terms composing \( H_{11} \) vanish, and a similar argument holds for \( H_{22} \). As a consequence, we have that \( H_{11}(s_1,s_2) + H_{22}(s_0,s_1) = 0 \) for every \( s_0, s_2 \), and therefore no topologically nontrivial invariant curve can exist.

\end{proof}

%\section{(To discuss) Integrability of the elliptic outer-length billiard} \label{ellisse}

\end{document}